\documentclass[11pt,a4paper,leqno]{amsart}

\usepackage{amsmath,amssymb,amsfonts,amsthm,enumitem,hyperref}
\setlist[enumerate]{label=(\roman*)}
\usepackage{tikz-cd}

\usepackage{hyperref}
\hypersetup{
	colorlinks = true,
	linkcolor = black,
	citecolor = black
}

\makeatletter
\@namedef{subjclassname@2020}{\textup{2020} Mathematics Subject Classification}
\makeatother

\newtheorem{thm}{Theorem}[section]
\newtheorem{lemma}[thm]{Lemma}
\newtheorem{prop}[thm]{Proposition}
\newtheorem{cor}[thm]{Corollary}
\newtheorem{question}{Question}
\theoremstyle{definition}
\newtheorem{df}[thm]{Definition}

\theoremstyle{remark}
\newtheorem{rem}[thm]{Remark}
\numberwithin{equation}{section}
\theoremstyle{plain}
\newcounter{theoremintro}
\newtheorem{theoremi}[theoremintro]{Theorem}

\DeclareMathOperator{\sr}{sr}
\DeclareMathOperator{\rr}{rr}
\DeclareMathOperator{\supp}{supp}
\DeclareMathOperator{\homeo}{Homeo}
\DeclareMathOperator{\prim}{Prim}
\DeclareMathOperator{\nucdim}{\dim_{\operatorname{nuc}}}

\newcommand{\Nb}{\mathbb{N}}
\newcommand{\Zb}{\mathbb{Z}}
\newcommand{\Cb}{\mathbb{C}}
\newcommand{\Tb}{\mathbb{T}}
\newcommand{\Rb}{\mathbb{R}}
\newcommand{\Zc}{\mathcal{Z}}
\newcommand{\Kc}{\mathcal{K}}

\begin{document}

\begin{abstract}
    We examine crossed product $C^*$-algebras associated with free non-minimal actions of countably infinite discrete abelian groups on the circle, extending the work of Putnam, Schmidt, and Skau. We obtain a large class of unital separable nuclear and non-simple $C^*$-algebras that are quasidiagonal, have stable rank one, and admit a unique tracial state. We determine their ideal structure and establish an improved uniform upper bound for their nuclear dimension. Finally, in the case $G = \Zb^d$, we compute the $K$-theory and trace pairing map.
\end{abstract}


 
\title[$C^*$-algebras associated with non-minimal actions on the circle]{Crossed product $C^*$-algebras associated with non-minimal actions on the circle}
\author{Jamie Bell}
\address{Jamie Bell, Mathematical Institute, University of Münster, Einsteinstrasse 62, 48149 Münster, Germany.}
\email{jbell@uni-muenster.de}

\thanks{Funded by the Deutsche Forschungsgemeinschaft (DFG, German Research Foundation) under Germany's Excellence Strategy EXC 2044/2 –390685587, Mathematics M\"unster: Dynamics-Geometry-Structure and project-ID 427320536, SFB 1442, of the DFG}

\maketitle

\section{Introduction}

The study of $C^*$-algebras has a profound connection with topological dynamics via the construction of crossed products, which translate group actions on spaces into operator-algebraic data. A central problem becomes to understand how properties of dynamics are reflected in their associated crossed products, and vice versa. Irrational rotation algebras, which arise from minimal integer actions on the circle, are prototypical of this fruitful interplay. Despite their straightforward underlying dynamics, they have, for instance, been instrumental to the development of Connes' noncommutative geometry programme \cite{Con94}, and in motivating early breakthroughs in the modern classification and structure theory for simple separable nuclear $C^*$-algebras. The striking success of the so-called Elliott classification programme continues to propel major developments in the study of $C^*$-algebras. 

While early classification efforts included non-simple algebras -- notably Bratteli--Elliott’s classification of AF $C^*$-algebras \cite{Bra72,Ell76} -- subsequent work, particularly as it relates to crossed products, has predominantly focussed on simple $C^*$-algebras. For an action $G\curvearrowright X$, every non-trivial closed $G$-invariant subset of $X$ gives rise to a non-trivial ideal of the crossed product $C_0(X) \rtimes G$. In particular, when studying simple crossed product $C^*$-algebras one must consider minimal (and usually free) actions, i.e., those without non-trivial closed invariant subsets, thereby ruling out natural examples such as those arising from hyperbolic and symbolic dynamics \cite{Nek22}. In this note, we investigate crossed products associated with certain non-minimal actions, i.e., non-simple crossed product $C^*$-algebras.

With a view toward possible applications to non-simple classification, we of course need to impose some restrictions, otherwise the general question of structure and classification is hopeless. A natural starting point for studying dynamical systems is to analyse their minimal sets (or, equivalently, their minimal subsystems), which roughly correspond to simple quotients in the crossed product.\footnote{For free actions of discrete amenable groups, which will shortly be our focus, this is indeed a bijective correspondence.} Although Zorn's lemma guarantees the existence of at least one minimal set for actions on compact spaces, minimal sets may be too abundant to describe collectively in any effective way. For instance, the minimal subsystems of the Bernoulli action $G\curvearrowright \{0,1\}^G$, known as minimal $G$-subshifts, form a vast uncountable family of dynamical systems which, while certainly an important class in their own right, are impractical to describe in any definitive manner. While certain special cases of non-minimal actions on the Cantor set and their $C^*$-algebras have been studied (see \cite{BezNiuSun21}), this rather suggests one should seek situations in which the number of minimal sets is manageable. In stark contrast with Cantor dynamics, the rigid topology of the circle supports very few minimal sets, all of which can be explicitly described (see Proposition~\ref{P-structure}), and this provides an important starting point for our investigations. For a more complete discussion on the subject of group actions on the circle, we refer to \cite{Ghy01,Nav11}.  

The second important consideration is which acting groups we allow. We shall consider actions of countable discrete groups, which, for actions on compact metrisable spaces, ensures that our crossed products are unital and separable. In this setting a common dividing line is that of amenability. $C^*$-algebraically, this grants access to nuclearity, which underpins much of the classification machinery. As such, we will also be interested in actions of amenable groups. At the technical level, too, we will see that amenability plays a crucial role, notably via Proposition~\ref{P-amen}. The second reason to focus on actions of amenable groups is that we will require at various points for our actions to be free, i.e., having no fixed points. By a classical result essentially due to H\"older \cite{Hol96} (see also \cite[Theorem 2.2.32]{Nav11}), if a discrete group admits a free action on the circle, then it must be abelian.  

Altogether, we turn rather inevitably to the study of crossed product $C^*$-algebras of free non-minimal actions of abelian groups on the circle. While such crossed products have been implicitly considered previously, an explicit study of their finer structural and regularity properties has rarely appeared in the literature. One notable exception is the work of Putnam, Schmidt, and Skau on Denjoy homeomorphisms \cite{PutSchSka86} that inspired the present paper, and which corresponds to the case of integer actions. As such, we refer to these actions as \emph{Denjoy actions} (Definition~\ref{D-Denjoy}). 

We establish several desirable properties for crossed products arising from free Denjoy actions, some extending \cite{PutSchSka86} and others entirely new. Just as the irrational rotation algebras (and their higher-dimensional brethren) have long inspired the study of simple $C^*$-algebras, we hope to advertise crossed products of Denjoy actions as rich non-simple counterparts, offering a similarly fertile ground for further exploration. 

Our main results are summarised below:

\begin{theoremi}\label{T-A}
	Let $G\curvearrowright \Tb$ be a free action of a countably infinite discrete abelian group. Then $A := C(\Tb) \rtimes G$ satisfies the following:
	\begin{enumerate} 
		\item $A$ has a unique trace;
		\item $A$ has stable rank one;
		\item $A$ is quasidiagonal; 
		\item $1\le \dim_{\operatorname{nuc}}(A) \le 3$, and $\dim_{\operatorname{nuc}}(A) = 1$ if $G\curvearrowright \Tb$ is minimal. 
	\end{enumerate} 
	If $G\curvearrowright \Tb$ is moreover non-minimal, then $A$ has a unique maximal ideal, isomorphic to a direct sum of countably many copies of $C_0(\Rb)\otimes \mathcal{K}(\ell^2(G))$, whose simple quotient is an approximately subhomogeneous $C^*$-algebra.
\end{theoremi}

The basic observation underlying much of Theorem~\ref{T-A} is that for an action $G\curvearrowright X$ and  minimal set $Y\subseteq X$, we obtain a $G$-equivariant quotient $C(Y)$ and ideal $C_0(X\setminus Y)$ of $C(X)$. If the action is moreover free, we thereby obtain a simple quotient $C(Y) \rtimes G$, and an ideal $C_0(X\setminus Y) \rtimes G$. In the case of free actions of discrete groups on the circle, there is a unique minimal set (which may be all of $\Tb$, in which case the ideal is trivial) and so $C(\Tb)\rtimes G$ can be understood in terms of the aforementioned quotient and ideal. This is also the approach taken in \cite{PutSchSka86}. As it turns out, when $G$ is abelian, the quotient is a well-understood stably finite Elliott-classifiable\footnote{That is, a simple, separable, unital and nuclear $C^*$-algebra satisfying the Universal Coefficient Theorem and tensorially absorbing the Jiang--Su algebra $\Zc$.} $C^*$-algebra, while the ideal is a continuous-trace $C^*$-algebra by Green's theorem \cite{Gre77}. This reflects the fact that $Y$ and $\Tb\setminus Y$ are, respectively, the non-wandering and wandering set for the action, in the sense of Birkhoff (cf.~\cite{Tom03}). With these structural results at hand, it is a matter of verifying the conditions required to lift regularity properties to the extension $C(\Tb)\rtimes G$. 

Unfortunately in the present note we do not realise the aim of classifying crossed products arising from free Denjoy actions, though we should point out that when $G = \Zb$ this is one of the main results of \cite{PutSchSka86}. By virtue of the structural description given in Theorem~\ref{T-A}, such a classification result would constitute a dovetailing of the recent classification results for simple nuclear $C^*$-algebras with the much earlier Dixmier--Douady classification of continuous-trace $C^*$-algebras \cite{DixDou63}. To achieve their classification result, Putnam, Schmidt, and Skau rely on Markley's dynamical classification of Denjoy homeomorphisms up to conjugacy \cite{Mar70} which, to the best of our knowledge, does not exist for actions of abelian groups. This approach parallels the classification of irrational rotation algebras \cite{Rie81,PimVoi80b}, where the angle of rotation is recovered via the pairing of a Rieffel projection in $K_0$ with the trace, which suffices to deduce isomorphism by Poincar\'e's classification theorem. It is noteworthy that the striking work of Giordano--Putnam--Skau (foreshadowed in \cite{PutSchSka86}) on Cantor minimal systems inverts this paradigm and derives dynamical classification from the abstract $C^*$-algebraic classification \cite{GioPutSka95}. 

Despite this, it is evident that $K$-theory will play an important role in obtaining further results about crossed products of Denjoy actions and so, restricting to actions of $\Zb^d$, we compute the $K$-theory and trace pairing map, again following the general outline presented in \cite{PutSchSka86}. 

\begin{theoremi}\label{T-Kthyi}
	Fix $d\in \Nb$. Let $\Zb^d \curvearrowright \Tb$ be a free Denjoy action with $\gamma_i$ the rotation number of the standard basis vector $e_i \in \Zb^d$, and let $A = C(\Tb) \rtimes \Zb^d$. Then, for $i\in \{0,1\}$,  
	\[
		K_i(A) \cong \Zb^{2^d}.
	\]	
	Moreover, denoting by $\tau$ the unique trace on $A$, there is an identification $K_0(A)\cong \Zb^{2^d}$ sending $[1_A]_0$ to the first standard basis vector, with respect to which the trace pairing $\tau_* : K_0(A) \to \Rb$ is given by 
	\[
		\tau_*(n_0,n_1,\ldots,n_{2^d - 1}) = n_0 + \sum_{i=1}^d n_i \gamma_i.	
	\]
	In particular,
	\[
		\tau_*(K_0(A)) = \Zb + \gamma_1\Zb + \cdots + \gamma_d \Zb.
	\] 
\end{theoremi}

For free minimal actions $\Zb^d\curvearrowright \Tb$, the crossed product $C(\Tb)\rtimes \Zb^d$ is a simple noncommutative torus. The $K$-theory of these algebras has been studied by Elliott \cite{Ell84} and Rieffel \cite{Rie88}, and they were later classified by Lin using their Elliott invariant \cite{Lin03}. Our approach to proving Theorem~\ref{T-Kthyi} is to identify a canonical embedding of a noncommutative torus in the crossed products appearing in Theorem~\ref{T-Kthyi} which we show induces an isomorphism of $K$-theory groups that respects the trace pairings (Theorems~\ref{T-Kthy} and \ref{T-Kthy2}). 

Finally, we would also like to point out connections between Theorem~\ref{T-Kthyi} and Bellissard's gap labelling conjecture \cite{Bel86}, and Matui's HK conjecture \cite{Mat15}. It would be interesting to determine if Theorem~\ref{T-Kthyi}, or possible extensions thereof, can be understood within these contexts. 

\medskip
\noindent {\it Acknowledgements.} I thank Shirly Geffen, Grigoris Kopsacheilis and Julian Kranz for their valuable feedback on earlier drafts, and Wilhelm Winter for inspiring discussions on nuclear dimension. I am grateful to the Department of Mathematical Sciences at Chalmers University of Technology and the University of Gothenburg for their hospitality during a research visit where parts of this work were carried out. I would also like to thank the anonymous referee for their careful reading of the manuscript and for providing many useful comments that improved the exposition and helped correct a number of inaccuracies.

\section{Preliminaries}\label{sec:background}

Let $G$ be a countable discrete group with identity $e$, and let $X$ be a locally compact Hausdorff space. A \emph{group action} $\alpha : G \curvearrowright X$ is a group homomorphism $\alpha : G \to \mathrm{Homeo}(X)$. When $\alpha$ is injective, i.e., the action is faithful, we identify $G$ as a subgroup of $\homeo(X)$ and simply write $gx$ for $\alpha(g)(x)$. The induced action on $C_0(X)$ is given by $(gf)(x) = f(g^{-1}x)$. For $x \in X$, the \emph{orbit} of $x$ is $Gx = \{gx : g \in G\}$ and the quotient $X/G$ of $X$ by the equivalence relation of belonging to the same orbit is called the \emph{orbit space}. A subset $Y \subseteq X$ is \emph{$G$-invariant} if $gY \subseteq Y$ for all $g \in G$. A non-empty closed invariant set is \emph{minimal} if it contains no proper non-empty closed invariant subset. The action is \emph{minimal} if $X$ is minimal (if $G$ is infinite, this is equivalent to all orbits being dense). The action is \emph{free} if $gx = x$ implies $g = e$ for all $x\in X$. Two actions $G \curvearrowright X$ and $G \curvearrowright Y$ are \emph{conjugate} if there exists a homeomorphism $\varphi : X \to Y$ such that $\varphi(gx) = g\varphi(x)$ for all $g \in G$ and $x\in X$. We write $M_G(X)$ for the set of $G$-invariant regular Borel probability measures on $X$. The action is \emph{uniquely ergodic} if $M_G(X)$ is a singleton. The \emph{support} of a measure $\mu\in M_G(X)$, denoted $\supp(\mu)$, is the intersection of all closed sets of full $\mu$-measure.

We write $\Tb$ for the unit circle, identified with $\Rb/\Zb$ via the quotient map $\pi : \Rb \to \Rb/\Zb$. A homeomorphism $f : \mathbb T \to \mathbb T$ is \emph{orientation-preserving} if it admits a lift $F : \mathbb R \to \mathbb R$ that is increasing and satisfies $F(x+1) = F(x)+1$ for all $x\in \Rb$. An action $G \curvearrowright \mathbb T$ is orientation-preserving if each $g \in G$ acts by an orientation-preserving homeomorphism.

\begin{prop}[{\cite[Proposition 5.6]{Ghy01}}]\label{P-structure}
	An orientation-preserving\footnote{Note that an orientation-reversing action is never free.} action $G \curvearrowright \mathbb T$ satisfies exactly one of the following:
	\begin{enumerate}
		\item it has a finite orbit;
		\item it is minimal;
		\item it admits a unique minimal set $Y$, homeomorphic to the Cantor set.
	\end{enumerate}
\end{prop}

\begin{df}\label{D-Denjoy}
	A \emph{Denjoy action} is an orientation-preserving, non-minimal action 
	$G \curvearrowright \mathbb T$ with all orbits infinite.
\end{df}

\begin{rem}
	We caution the reader that it is common, especially in the $C^*$-algebra literature, to use Denjoy homeomorphism to refer to an action of the integers on the Cantor set obtained by restricting a non-minimal homeomorphism to its unique minimal set. We reserve the term Denjoy for actions on the circle. 
\end{rem}

A \emph{monotone degree one map} is a continuous surjection 
$\varphi : \mathbb T \to \mathbb T$ that admits a non-decreasing lift 
$\tilde{\varphi} : \mathbb R \to \mathbb R$ with 
$\tilde{\varphi}(x+1) = \tilde{\varphi}(x) + 1$ for all $x\in\Rb$. Given actions $\alpha, \beta : G \curvearrowright \mathbb T$, a \emph{semiconjugacy} from $\alpha$ to $\beta$ is a monotone degree one map $\varphi$ such that, for all $g\in G$, we have
\[
\beta(g) \circ \varphi = \varphi \circ \alpha(g). 
\]
Every Denjoy action of a countably infinite group is semiconjugate to a minimal action. Since a description of the semiconjugation map will be useful later, we supply a proof (see also \cite[Proposition 5.8]{Ghy01}). 

\begin{prop}\label{P-semiconj}
    Let $G\curvearrowright\Tb$ be a Denjoy action of a countably infinite discrete group. Then $G\curvearrowright \Tb$ is semiconjugate to a minimal action.
\end{prop}

\begin{proof}
    Let $Y\subsetneq \Tb$ be the unique minimal set  as in Proposition~\ref{P-structure}. Then $\Tb\setminus Y = \bigcup_{n\in \Nb} I_n$ where the $I_n$ are open intervals with pairwise disjoint closures (since $Y$ has no isolated points). Define an equivalence relation $\sim$ on $\Tb$ by $x\sim y$ if and only if there exists $n\in \Nb$ such that $\{x,y\}\subseteq \overline{I_n}$. Letting $\varphi : \Tb \to \Tb/{\sim}$ be the quotient map, it is readily checked that $\Tb/{\sim}$ is homeomorphic to $\Tb$ and that $\varphi$ is a monotone degree one map. Indeed, one may take $\tilde{\varphi} : \Rb \to \Rb$ to be the Cantor function \cite[\S 2.7]{RoyFit10} (extended in the obvious manner to a function on $\Rb$), which is constant on each interval of $\Tb\setminus Y$. Since $Y$ is invariant, $G\curvearrowright \Tb$ descends to a well-defined action on $\Tb/{\sim}$ which intertwines the map $\varphi$. The induced action on $\Tb/{\sim}$ is minimal since orbits of points in $\Tb$ are dense in $Y$. 
\end{proof}  

Let $f : \Tb \to \Tb$ be orientation-preserving, and $F : \Rb \to \Rb$ a lift of $f$. Then 
\[
\rho(f) := \lim_{n\to\infty} \frac{F^n(x)}{n} \mod \mathbb Z
\]
exists and is independent of $x$ and the lift. We call $\rho(f) \in \mathbb T$ the \emph{rotation number} of $f$, introduced by Poincar\'e \cite[p.~233]{Poi90}. We often identify $\rho(f)$ via some lift with a real number $\rho(f)\in [0,1)$. For an orientation-preserving group action $G\curvearrowright \mathbb{T}$, we consider the map $\rho : G \to \mathbb{T}$ assigning $g\in G$ to its corresponding rotation number. While $\rho$ is not a homomorphism in general, it is for actions of amenable groups. In fact much more is true: 

\begin{prop}[{\cite[Proposition 6.17]{Ghy01}}]\label{P-amen}
	Let $G$ be an amenable group and $G\curvearrowright \Tb$ an orientation-preserving action. Then the rotation number map $\rho : G \to \mathbb{T}$ is a group homomorphism. If $\rho(G)$ is finite, then $G\curvearrowright \mathbb{T}$ has a finite orbit. Otherwise, the action is semiconjugate to the minimal rotation action given by $g\mapsto R_{\rho(g)}$, where $R_{\rho(g)}$ denotes the rotation homeomorphism by $\rho(g)$. This semiconjugacy is a conjugacy if and only if $G\curvearrowright \Tb$ is minimal. 
\end{prop}

\section{The results}

\subsection{A short exact sequence}\label{sec:extension}

We introduce the fundamental extension of $C^*$-algebras associated with a Denjoy action, following the approach outlined in the introduction. 

Let $G$ be a countably infinite discrete group, and let $G\curvearrowright \Tb$ be a Denjoy action on the circle with unique minimal set $Y\subsetneq \Tb$. Since $Y$ is $G$-invariant, we obtain a $G$-equivariant short exact sequence 
\begin{equation}
    0 \longrightarrow C_0(\Tb\setminus Y) \stackrel{\tilde{\iota}}{\longrightarrow} C(\Tb) \stackrel{\tilde{q}}{\longrightarrow} C(Y) \longrightarrow 0\label{exact_1}
\end{equation}
where $\tilde{\iota}$ is the inclusion map and $\tilde{q}$ the restriction map. By definition, when $G$ is exact (note that all amenable groups are exact) (\ref{exact_1}) yields a short exact sequence of $C^*$-algebras 
\begin{equation} 
    0 \longrightarrow C_0(\Tb\setminus Y)\rtimes_r G \stackrel{\iota}{\longrightarrow} C(\Tb)\rtimes_r G \stackrel{q}{\longrightarrow} C(Y)\rtimes_r G \longrightarrow 0.\label{exact_2}
\end{equation}
The map $\iota$ is the natural inclusion map coming from $\tilde{\iota}$ and $q$ is the quotient map satisfying $q(f\lambda_g) = \tilde{q}(f)\lambda_g$ for $f\in C(\Tb)$ and $g\in G$, where $\lambda_g$ is the canonical unitary in $C(\Tb)\rtimes G$. Whenever $G$ is amenable, we identify the reduced and full crossed product $C^*$-algebras via the canonical isomorphism. 

Proposition~\ref{P-amen} allows for a precise description of the dynamics of a Denjoy action in the case when the acting group is abelian (cf.\ \cite[\S 3]{PutSchSka86}). This leads to an explicit description of the ideal appearing in (\ref{exact_2}). 

\begin{df}\label{D-proper}
    Let $G\curvearrowright X$ be an action of a discrete group $G$ on a locally compact Hausdorff space $X$. The action is \emph{proper} if for every non-empty compact subset $K\subseteq X$, the set $G_K = \{g\in G : gK\cap K \ne \emptyset\}$ is finite. 
\end{df}

\begin{prop}\label{P-proper}
    Let $G$ be a countably infinite discrete abelian group, let $G\curvearrowright \Tb$ be a free Denjoy action, and let $Y\subseteq \Tb$ be its unique minimal set. Then the restricted action $G\curvearrowright \Tb\setminus {Y}$ is free and proper. 
\end{prop}

\begin{proof}
Since $G$ is amenable, by Proposition~\ref{P-amen}, $G\curvearrowright \Tb$ is semiconjugate to the minimal action $g\mapsto R_{\rho(g)}$. Moreover, freeness implies that $\rho$ is injective. In fact, $\rho(g) = 0$ if and only if the $\Zb$-action associated with $g$ has a fixed point, so that for a free action, $\rho(g) = 0$ implies $g = e$. Fix an open interval $I\subseteq \Tb \setminus Y$. We first show that 
\begin{equation}
	g I\cap I = \emptyset \quad \text{for all } g\in G\setminus \{e\}. \label{E-cov}
\end{equation} 
By the proof of Proposition~\ref{P-semiconj}, there exists a monotone degree one map $\varphi : \Tb\to\Tb$ sending $I$ to some point $y\in \Tb$. By Proposition~\ref{P-amen} and the definition of $\varphi$, we have $\varphi(gI) = R_{\rho(g)}(y) = \{y + \rho(g)\}$, so by injectivity of $\rho$ we have $y + \rho(g) \ne y$ for all $g\in G\setminus \{e\}$. We deduce that their preimages $gI = \varphi^{-1}(y+\rho(g))$ and $I = \varphi^{-1}(y)$ are disjoint. 

Now let $K\subseteq\mathbb T\setminus Y$ be compact. Since $\mathbb T\setminus Y = \bigcup_{n\in\Nb} I_n$ is covered by open intervals, compactness yields a finite subcover of $K$, say $K\subseteq\bigcup_{j=1}^m I_{n_j}$. Suppose $g\in G$ satisfies $gK\cap K\ne \emptyset$. Then there exist $i,j\in \{1,\dots,m\}$ and points $x\in K\cap I_{n_i}$, $y\in K\cap I_{n_j}$ with $gx=y$. Consequently $gI_{n_i}\cap I_{n_j}\neq \emptyset$ and it follows that $gI_{n_i}=I_{n_j}$. For each ordered pair $(i,j)$, by (\ref{E-cov}) we have $|\{g\in G : gI_{n_i} = I_{n_j}\}| \le 1$. Therefore $G_K\subseteq \bigcup_{i,j=1}^m\{g\in G : gI_{n_i} = I_{n_j}\}$ has cardinality at most $m^2$. Thus the action $G\curvearrowright\mathbb T\setminus Y$ is proper. 

Finally, the action $G\curvearrowright \Tb\setminus Y$ is free, being the restriction of a free action. 
\end{proof}

\begin{cor}\label{P-ideal}
    Let $G$ be a countably infinite discrete abelian group and let $G\curvearrowright \Tb$ be a free Denjoy action. Then the ideal $C_0(\Tb\setminus Y)\rtimes G$ as in (\ref{exact_2}) is isomorphic to $\bigoplus_{i=1}^{k} C_0(\Rb)\otimes \Kc(\ell^2(G))$, where $k \in \Nb \cup \{\infty\}$ is the number of (necessarily disjoint) orbits of open intervals in $\mathbb{T}\setminus Y$.
\end{cor}

\begin{proof}
    Since $Y\subseteq \Tb$ is invariant, every interval in $\Tb\setminus Y$ is mapped onto another. Let $k \in \Nb\cup \{\infty\}$ be the number of orbits of intervals in $\Tb\setminus Y$. Let $I_{n_1},\ldots,I_{n_k}$ be representatives for each of these orbits of intervals, and put $Y_0 := \bigcup_{i = 1}^k I_{n_i}$. Evidently, $Y_0 \cong (\Tb\setminus Y)/G$ and so by Proposition~\ref{P-proper}, we may apply Green's theorem \cite[Corollary 15]{Gre77} to obtain an isomorphism $C_0(\Tb\setminus Y) \rtimes G \cong C_0(Y_0)\otimes \Kc(\ell^2(G))$. Since the intervals $I_{n_i}$ belong to disjoint orbits, we have 
    \[
        C_0(\Tb\setminus Y) \rtimes G \cong \bigoplus_{i=1}^k C_0(I_{n_i}) \otimes \Kc(\ell^2(G)) \cong \bigoplus_{i=1}^k C_0(\Rb)\otimes \Kc(\ell^2(G)), 
    \]
which completes the proof. 
\end{proof}

\begin{rem}\label{R-quotient}
	Using classification machinery, the structure of the quotient appearing in (\ref{exact_2}) may also be described rather explicitly. Indeed, it is a simple separable unital nuclear $C^*$-algebra that is isomorphic to an inductive limit of homogeneous $C^*$-algebras of topological dimension at most two (see, e.g., \cite[Theorem 8.2]{KerSza20}). In the special case $G = \mathbb{Z}$, it is isomorphic to an A$\Tb$ algebra \cite{Put89}.  
\end{rem}

\subsection{Traces and ideal structure}\label{sec:trace_ideal}

In this section we prove that crossed products of free actions of abelian groups on the circle have a unique tracial state, and determine the ideal structure for free Denjoy actions. By a trace we mean a tracial state and an ideal refers to a closed two-sided ideal. 

\begin{prop}\label{P-unique erg}
    Every free action $G\curvearrowright \Tb$ of a countably infinite discrete abelian group is uniquely ergodic. 
\end{prop}

\begin{proof}
    By Proposition~\ref{P-structure} it suffices to consider the following two cases. In both cases, existence of at least one probability measure $\mu \in M_G(\Tb)$ is guaranteed by amenability of $G$, and so it suffices to show uniqueness. 
    
    If $G\curvearrowright \Tb$ is minimal, then by Proposition~\ref{P-amen}, $G\curvearrowright \Tb$ is conjugate to an action by rotations. Furthermore, the set of angles of the rotation is dense in $\Tb$ by minimality. Clearly this implies that the Lebesgue measure is the unique invariant measure. Since conjugacy preserves unique ergodicity, the original action $G\curvearrowright \Tb$ is uniquely ergodic. 

    If $G\curvearrowright \Tb$ is not minimal, it must be a free Denjoy action with unique minimal set $Y\subsetneq \Tb$. To prove uniqueness, let $\nu\in M_G(\Tb)$. Proposition~\ref{P-semiconj} provides a semiconjugacy $\varphi : \Tb \to \Tb$ to a minimal action $\beta : G\curvearrowright \Tb$. The pushforward measure $\varphi_*\nu$ is $\beta$-invariant. As $\beta$ is minimal and each $\beta(g)$ acts by irrational rotations, $\varphi_*\nu = \lambda$. By the continuity and monotonicity of $\varphi$, its fibres, i.e., preimages $\varphi^{-1}(z)$ of points $z\in \Tb$ are either singletons or closed intervals on which $\varphi$ is constant. Let $I$ be a connected component of $\Tb\setminus Y$. By (\ref{E-cov}), the sets $\{gI : g\in G\}$ are disjoint. By $G$-invariance of $\nu$, this implies that $\nu(\bigcup_{g\in G} gI) = \sum_{g\in G} \nu(gI) = \sum_{g\in G} \nu(I)$, which is unbounded when $G$ is infinite unless $\nu(\Tb\setminus Y) = 0$. Thus $\supp(\nu) \subseteq Y$. The map $\varphi|_Y : Y\to\Tb$ is a Borel isomorphism and so since $\varphi_*\mu = \lambda = \varphi_*\nu$ and both $\mu$ and $\nu$ are supported on $Y$, we deduce that $\mu = (\varphi|_Y)_*^{-1}\lambda = \nu$. 
\end{proof}

\begin{rem}\label{R-meas}
	For Denjoy actions $G\curvearrowright \mathbb{T}$, the unique invariant measure $\mu$ in Proposition~\ref{P-unique erg} can be described explicitly. Indeed, $\mu$ is the image under the quotient map $\pi : \Rb \to \Tb$ of the Lebesgue--Stieltjes measure for the map $\tilde{\varphi} : \mathbb{R} \to \mathbb{R}$ corresponding to the semiconjugacy $\varphi : \Tb \to \Tb$. It follows that for every $g\in G,z\in \Tb$, we have $\mu((z,gz]) = \rho(g)$ (see \cite[Theorem 2.2.10]{Nav11}).   
\end{rem}

\begin{lemma}\label{L-minimal supp}
    Let $G\curvearrowright X$ be an action of a discrete amenable group on a compact Hausdorff space $X$ and suppose that $M_G(X) = \{\mu\}$. Then $\supp(\mu)$ is the unique minimal set. In particular, $G\curvearrowright X$ is minimal if and only if $\mu$ has full support.  
\end{lemma}

\begin{proof}
     Since $\supp(\mu)$ is a non-empty closed (hence compact) and $G$-invariant set, it contains a minimal set $Y\subseteq \supp(\mu)$. Since $G$ is amenable, there exists a Borel measure $\nu \in M_G(Y)\subseteq M_G(X)$ such that $\nu(Y) = 1$. By unique ergodicity, we have $\nu = \mu$ and hence $\mu(Y) = 1$. It follows from the definition of $\supp(\mu)$ that $\supp(\mu)\subseteq Y$, hence $\supp(\mu) = Y$ is a minimal set. If $Y_1,Y_2\subseteq X$ are two minimal sets then since $Y_1\cap Y_2$ is also closed and $G$-invariant, either $Y_1 = Y_2$ or $Y_1$ and $Y_2$ are disjoint. If $Y_1$ and $Y_2$ are disjoint, then $\mu_i := \mu|_{Y_i}$ determine two distinct probability measures, contradicting uniqueness of $\mu$. Hence $\supp(\mu)$ is the unique minimal set. 
\end{proof}

\begin{cor}\label{T-trace}
    Let $G\curvearrowright \Tb$ be a free action of a countably infinite discrete abelian group. Then $C(\Tb)\rtimes G$ has a unique tracial state $\tau$ satisfying 
    \begin{equation}
        \tau(f\lambda_g) = \begin{cases}
            \displaystyle\int_{\Tb} f\, d\mu, & g = e\\
            0, & \text{otherwise}
        \end{cases} \label{trace_eqn}
    \end{equation}
    where $\mu$ is the unique measure in $M_G(\Tb)$, $f\in C(\Tb)$ and $g\in G$. Moreover, the support of $\mu$ coincides with the unique minimal set for the action and in particular, $\tau$ is faithful if and only if the action is minimal. 
\end{cor}

\begin{proof}
    By Proposition~\ref{P-unique erg} and Lemma~\ref{L-minimal supp}, $M_G(\Tb) = \{\mu\}$ and $\supp(\mu)$ is the unique minimal set. By \cite[Corollary 2.8]{KawTakTom90}, tracial states on $C(\Tb)\rtimes G$ correspond bijectively to $G$-invariant probability measures, so there is a unique tracial state $\tau$ of the form (\ref{trace_eqn}). Finally, $\tau$ is faithful if and only if $\mu$ has full support which, since $\supp(\mu)$ is the unique minimal set, holds if and only if the action is minimal.  
\end{proof} 

\begin{cor}\label{C-quotient trace}
     Let $G\curvearrowright \Tb$ be a free Denjoy action of a countably infinite discrete abelian group. Then the quotient $C(Y)\rtimes G$ as in (\ref{exact_2}) has a unique tracial state $\tau'$ satisfying $\tau = \tau'\circ q$, where $\tau$ is the unique tracial state on $C(\Tb)\rtimes G$. 
\end{cor}

The following generalises \cite[Proposition 4.1]{PutSchSka86}, with the same proof.

\begin{prop}
    Let $G\curvearrowright \Tb$ be a free Denjoy action of a countably infinite discrete abelian group. Then the ideal $C_0(\Tb\setminus Y)\rtimes G$ as in (\ref{exact_2}) is the unique maximal ideal in $C(\Tb)\rtimes G$, and corresponds to the trace ideal $\{a \in C(\Tb)\rtimes G : \tau(a^*a) = 0\}$ of the unique tracial state $\tau$ on $C(\Tb)\rtimes G$. 
\end{prop}

\begin{proof}
    By \cite[Theorem 5.15]{Zel68}, every ideal in $A = C(\Tb)\rtimes G$ has the form $C_0(\Tb\setminus Z)\rtimes G$ for a closed, $G$-invariant subset $Z$ of $\Tb$. Since $Y$ is the unique minimal set for $G\curvearrowright \Tb$ (Proposition~\ref{P-structure}), $J := C_0(\Tb\setminus Y)\rtimes G$ is the unique maximal ideal. Let $J_\tau = \{a\in A : \tau(a^*a) = 0\}$. By Corollary~\ref{T-trace} $\supp(\mu) = Y$, thus for $f\in C_0(\Tb\setminus Y)$ and $g\in G$, we have 
    \[
        \tau((f\lambda_g)^*(f\lambda_g)) = \int_{\Tb} |f|^2 \, d\mu = 0.
    \]
    Since elements of the form $f\lambda_g$ span a dense subalgebra of $J$, we have $J\subseteq J_\tau$. Conversely, since $J_\tau$ is an ideal in $A$, we have $J_\tau\subseteq J$ by uniqueness of $J$, so $J_\tau = J$. 
\end{proof}

An ideal is primitive if it is the kernel of an irreducible representation. The set of primitive ideals, $\prim(A)$, is equipped with the hull-kernel topology as determined by the closure operation: the closure of a non-empty subset $F\subseteq \prim(A)$ is $\overline{F} := \{P \in \prim(A) : \bigcap_{I\in F} I \subseteq P\}$. From Corollary~\ref{P-ideal}, we obtain the analogue of \cite[Proposition 4.4]{PutSchSka86}. 

\begin{prop}\label{P-prim}
    Let $G\curvearrowright \Tb$ be a free Denjoy action of a countably infinite discrete abelian group on the circle, and let $A = C(\Tb)\rtimes G$. Then  $\prim(A)$ is homeomorphic to $Y_0\cup \{J\}$, where $Y_0 = \bigcup_{i=1}^k I_{n_i}$ is a union of representative intervals for each orbit of intervals in $\Tb\setminus Y$ with the subspace topology, and $\{J\}$ is contained in the closure of every non-empty subset. 
\end{prop}

\subsection{Stable rank one and quasidiagonality}\label{sec:stable_quasi}

In this section, we prove stable rank one and quasidiagonality for crossed products of free  actions of abelian groups on the circle. Recall that a unital $C^*$-algebra has \emph{stable rank one} if the set of invertible elements is dense in norm. A $C^*$-algebra $A$ is said to be \emph{quasidiagonal} if it admits a faithful representation $\pi : A \to B(H)$ such that $\pi(A)$ is a family of quasidiagonal operators. 

\begin{thm}\label{T-sr}
    Let $G\curvearrowright \Tb$ be a free action of a countably infinite discrete abelian group on the circle. Then $C(\Tb)\rtimes G$ has stable rank one. 
\end{thm}

\begin{proof}
    By Proposition~\ref{P-structure}, there are two cases: if $G\curvearrowright \Tb$ is free and minimal, then $C(\Tb)\rtimes G$ is simple unital and $\Zc$-stable (see, e.g., \cite[Theorem 8.2]{KerSza20}). Since it also admits a trace (Corollary~\ref{T-trace}), it must be finite and hence has stable rank one by \cite[Theorem 6.7]{Ror04}. 
    
    If $G\curvearrowright \Tb$ is not minimal, then it is a free Denjoy action. Let $A = C(\Tb)\rtimes G$, $J = C_0(\Tb\setminus Y)\rtimes G$ and $B = C(Y)\rtimes G$. Applying the same argument as the previous paragraph, $B$ has stable rank one. By Corollary~\ref{P-ideal}, $J \cong \bigoplus_{i=1}^k C_0(\Rb)\otimes \Kc(\ell^2(G))$. Hence 
    \begin{equation}
        \sr(J) = \sr\Big(\bigoplus_{i=1}^k C_0(\Rb)\otimes \Kc(\ell^2(G))\Big) = \sr(C_0(\Rb)\otimes \Kc(\ell^2(G))). \label{sr_calc}
    \end{equation}
    Since the minimal unitisation of $C_0(\Rb)$ is isomorphic to $C(\Tb)$ and $\dim \Tb = 1$, we have $\sr(C_0(\Rb)) = 1$ by \cite[Proposition 1.7]{Rie83}. Combining this with basic permanence properties of stable rank one, we conclude that $\sr(J) = 1$. By \cite[Lemma 3]{Nis87}, it remains to show that the index map $K_1(B) \to K_0(J)$ vanishes, which is immediate from the fact, easily verified by standard $K$-theoretic considerations, that $K_0(J) = \{0\}$.  
\end{proof}

\begin{rem}
    Freeness is a necessary assumption in Theorem~\ref{T-sr}. Indeed the trivial action of $\Zb^d$ on $\Tb$ yields a crossed product that is isomorphic to $C(\Tb^{d+1})$, which does not have stable rank one for all $d\ge 1$. 
\end{rem}

\begin{thm}\label{T-quasidiag}
    Let $G\curvearrowright \Tb$ be a free action of a countably infinite discrete abelian group on the circle. Then $C(\Tb)\rtimes G$ is quasidiagonal.
\end{thm}

\begin{proof}
    If $G\curvearrowright \Tb$ is minimal, then $C(\Tb)\rtimes G$ is separable, nuclear and satisfies the UCT \cite{Tu99}, and moreover admits a faithful tracial state by Corollary~\ref{T-trace}. Hence it is quasidiagonal by \cite{TikWhiWin17}. In the non-minimal case, we again recall the short exact sequence (\ref{exact_2}). Let $J$ and $B$ denote the ideal and quotient, respectively. Since $C(\Tb)\rtimes G$ is separable and $B$ is nuclear and satisfies the UCT, by \cite[Theorem 3.4]{BroDad04} it suffices to prove that $J$ and $B$ are quasidiagonal and the index map $K_1(B) \to K_0(J)$ vanishes. The ideal $J$ is quasidiagonal since every commutative $C^*$-algebra is quasidiagonal and quasidiagonality is preserved under stabilisation. By \cite{TikWhiWin17}, since $B$ is nuclear, satisfies the UCT, and has a faithful trace (Corollary~\ref{T-trace}), $B$ is also quasidiagonal. Finally, the index map vanishes since $K_0(J) = \{0\}$.
\end{proof}

\begin{rem}
	We have been unable to verify the stronger statement that $C(\Tb)\rtimes G$ in fact embeds into an AF algebra. When the action is minimal, this follows from Corollary~\ref{T-trace} and the main result of \cite{Sch20}. For the same reason, the quotient $C(Y)\rtimes G$ embeds into a unital simple AF algebra. We do not know of an analogue of \cite{Pim83}, which holds for possibly non-simple crossed products $C(X)\rtimes \Zb$, but suspect that $C(\Tb)\rtimes G$ should always embed into a (not necessarily simple) AF algebra (cf.~\cite[Proposition 4.3]{PutSchSka86}).   
\end{rem}

\subsection{Further applications}\label{sec:applications}

Introduced by Winter and Zacharias, nuclear dimension is a generalisation of covering dimension for topological spaces to (nuclear) $C^*$-algebras \cite{WinZac10}. Finite nuclear dimension has come to play a crucial role in the Elliott classification programme. Crossed products arising from free actions of finitely-generated countably infinite abelian groups on finite-dimensional spaces are known to have finite nuclear dimension thanks to the work of Szab\'o, Wu and Zacharias \cite[Corollary 8.6]{SzaWuZac19}. The authors obtain the upper bound 
\begin{equation}
    \nucdim(C(\Tb)\rtimes G) \le 4\cdot 3^{h(G)} - 1, \label{nuc_dim_est}
\end{equation}
where $h(G)$ denotes the Hirsch length of $G$ (see also \cite{Sza15}). Recall that $h(\Zb^d) = d$, so the upper bound in (\ref{nuc_dim_est}) can become arbitrarily large. 
 
\begin{cor}\label{C-nuc}
    Let $G\curvearrowright\Tb$ be a free action of a countably infinite discrete abelian group on the circle. Then 
    \begin{equation*}
        1\le \nucdim(C(\Tb)\rtimes G) \le 3.
    \end{equation*}
	If $G\curvearrowright \Tb$ is moreover minimal, then $\nucdim(C(\Tb)\rtimes G) = 1$. 
\end{cor}

\begin{proof}
    If $G\curvearrowright \Tb$ is minimal, then $\nucdim(C(\Tb)\rtimes G) \le 1$ by \cite[Theorem B]{Casetal21} and we claim that $\nucdim(C(\Tb)\rtimes G) \ge 1$. It suffices to show that $C(\Tb)\rtimes G$ is not AF by \cite[Remark 2.2~(iii)]{WinZac10}. To see this, note that $G\curvearrowright \Tb$ is conjugate to the rotation action determined by the rotation homomorphism $\rho : G \to \Tb$ (Proposition~\ref{P-amen}). By the universal property of the full crossed product, $C(\Tb)\rtimes G$ is isomorphic to $C^*(G)\rtimes_\alpha \Zb$, where $\alpha : \Zb \curvearrowright C^*(G)$ acts via the automorphism $\lambda_g \mapsto \overline{\rho(g)}\lambda_g, g\in G$. The trivial representation of $G$ induces a unital $*$-homomorphism $\nu : C^*(G) \to \Cb$. Hence the induced map on $K_0$ satisfies $\nu_*([\lambda_{e}]_0) = [1]_0 \ne 0$. It follows from the Pimsner--Voiculescu six-term exact sequence \cite{PimVoi80} that $C^*(G)\rtimes_\alpha \Zb \cong C(\Tb)\rtimes G$ has non-trivial $K_1$. Indeed, since $\alpha$ is unital, the class $[\lambda_e]_0$ lies in the kernel of $\mathrm{id}-\alpha_* \colon K_0(C^*(G))\to K_0(C^*(G))$, so by exactness there is an element of $K_1(C^*(G)\rtimes \mathbb Z)$ mapping to $[\lambda_e]_0 \ne 0$ under the boundary map. Thus $K_1(C(\mathbb T)\rtimes G)\neq 0$, and since AF algebras have trivial $K_1$, this proves the claim. Hence $\nucdim(C(\Tb)\rtimes G) = 1$. 
    
    For $G\curvearrowright \Tb$ non-minimal, recall the short exact sequence (\ref{exact_2}) given by 
    \[
        0 \longrightarrow J \longrightarrow A \longrightarrow B \longrightarrow 0
    \]
    where $J\cong \bigoplus_{i=1}^k C_0(\Rb)\otimes \Kc$ by Corollary~\ref{P-ideal}, $A = C(\Tb)\rtimes G$ and $B$ is unital separable simple nuclear and $\Zc$-stable \cite{Ker20} (cf.~Remark~\ref{R-quotient}). Using \cite[Corollary 2.10]{WinZac10}, we have $\nucdim(J) = \dim(\Rb) = 1$ and $\nucdim(B) \le 1$ by \cite[Theorem B]{Casetal21}. Hence it follows from \cite[Proposition 2.9]{WinZac10} that 
    \[
        \nucdim(A) \le \nucdim(J) + \nucdim(B) + 1 \le 3.
    \]
    Since $\nucdim(J) = 1$, we also have $\nucdim(A)\ge \nucdim(J) \ge 1$, which completes the proof. 
\end{proof}

We point out that the finite-generation assumption made in \cite{SzaWuZac19} can be dispensed with in Corollary~\ref{C-nuc}. Obtaining sharp estimates for nuclear dimension is notoriously difficult, though, in light of Corollary~\ref{C-nuc}, we pose the following:

\begin{question}
    Let $G\curvearrowright \Tb$ be a free Denjoy action. What is $\nucdim(C(\Tb)\rtimes G)$? 
\end{question}

\begin{rem}
	The Toms--Winter conjecture (which is known to hold in the monotracial case \cite{SatWhiWin15}) predicts that for simple separable unital nuclear $C^*$-algebras, finite nuclear dimension is equivalent to tensorial absorption of the Jiang--Su algebra $\mathcal{Z}$. However, if $G \curvearrowright \mathbb{T}$ is a free Denjoy action, then $C(\mathbb{T}) \rtimes G$ is not $\mathcal{Z}$-stable. The obstruction comes from the presence of a nonzero type~I ideal, which prevents absorption of the Jiang--Su algebra. 
\end{rem}

The real rank of a $C^*$-algebra $A$, denoted $\rr(A)$, was introduced by Brown and Pedersen \cite{BroPed91} and is defined similarly to stable rank. It follows from the definition that $\rr(A) \le 2\sr(A) - 1$ (see \cite[Proposition 1.2]{BroPed91}).

\begin{cor}\label{C-rr}
    Let $G \curvearrowright \Tb$ be a free Denjoy action of a countably infinite discrete abelian group. Then $C(\Tb)\rtimes G$ has real rank one.
\end{cor}

\begin{proof}
    Let $A = C(\Tb)\rtimes G$. By Theorem~\ref{T-sr} and the comment preceding the corollary, $\rr(A)\le 1$. By \cite[Th\'eor\`eme 1.4]{Has95}, we have $\rr(A)\ge \rr(J)$, where $J = C_0(\Tb\setminus Y)\rtimes G$ is the unique maximal ideal in $A$. However by Corollary~\ref{P-ideal}, $\rr(J) = \rr(C_0(\Rb)\otimes \Kc) \ge \rr(C_0(\Rb)) = 1$ (the latter equality follows from \cite[Theorem 2.5]{BroPed91}). 
\end{proof}

\begin{rem}
	In contrast to Corollary~\ref{C-rr}, for a free minimal action $\Zb^d \curvearrowright \Tb$ the crossed product $C(\Tb)\rtimes \Zb^d$ is isomorphic to a simple noncommutative torus (see Section~\ref{S-Kthy}) and hence has real rank zero \cite[Theorem 3.8]{Phi06}. 
\end{rem}

\subsection{$K$-theory} \label{S-Kthy}

In this section we prove Theorem~\ref{T-Kthyi} (Theorems~\ref{T-Kthy} and \ref{T-Kthy2}). We refer the reader to \cite{RorLarLau00} for the necessary background on $K$-theory for $C^*$-algebras. For a projection $p\in M_\infty(A)$, we write $[p]_0$ for its class in $K_0(A)$ and for a unitary $u\in M_\infty(A)$ we write $[u]_1$ for its class in $K_1(A)$ (note that all of the $C^*$-algebras in this section are unital). 

As mentioned in the introduction, the proof of Theorem~\ref{T-Kthyi} involves finding a canonical embedding of a simple noncommutative torus in the crossed product associated with a free Denjoy action. We show that this inclusion induces an isomorphism of $K$-groups that is also compatible with the trace pairing on $K_0$. This idea is already implicit in \cite{PutSchSka86} (see the remark after the proof of Lemma 6.4 therein), though not stated as such. This allows us to outsource a lot of the technical details to \cite{Ell84} and, in particular, provides a more convenient description of a basis for the $K$-theory (see Remark~\ref{R-gen}).   

By Proposition~\ref{P-amen}, a free minimal action $\Zb^d\curvearrowright \Tb$ is given up to conjugacy by rotations. Thus $C(\Tb) \rtimes \Zb^d$ is generated by the unitary generator $u_1$ of $C(\Tb)$, given explicitly by $u_1(z) = z$, and unitaries $u_2,\ldots,u_{d+1}$ implementing the action. For each $i \in \{2,\ldots,d+1\}$, we have $u_i u_1 u_i^* = e^{-2\pi i \gamma_{i-1}}u_1$, where $\gamma_i$ is the angle of rotation of $e_i\in \Zb^d$, and hence rearranging slightly we have 
\[
	C(\Tb) \rtimes \Zb^d \cong C^*(u_1,\ldots,u_{d+1} : u_ju_i = e^{2\pi i \theta_{i,j}}u_iu_j), 
\]
where $\theta = (\theta_{i,j})$ is the real skew-symmetric matrix 
\begin{equation}
	\theta := \begin{pmatrix}
		0 & \gamma_1 & \gamma_2 & \cdots & \gamma_d \\
		-\gamma_1 & 0 & 0 & \cdots & 0\\
		-\gamma_2 & 0 & 0 & \cdots & 0\\
		\vdots & \vdots & \vdots & \ddots & \vdots \\
		-\gamma_d & 0 & 0 & \cdots & 0
	\end{pmatrix}. \label{E-torus}
\end{equation}
The matrix $\theta$ is \emph{nondegenerate}, i.e., if $x\in \Zb^{d+1}$ satisfies $e^{2\pi i\langle x, \theta y\rangle} = 1$ for all $y\in \Zb^{d+1}$, then $x=0$; this reflects the simplicity of $C(\Tb)\rtimes \Zb^d$ (see, e.g., \cite{Phi06}). In summary, $C(\Tb)\rtimes \Zb^d$ is isomorphic to the simple $(d+1)$-dimensional noncommutative torus $A_\theta$. Every noncommutative torus is equipped with a \emph{canonical trace} $\tau_\theta$, given by composing the conditional expectation onto $C(\Tb)$ with integration against the normalised Lebesgue measure on $\Tb$. For matrices of the form (\ref{E-torus}), this trace is unique.  

We recall for the reader's convenience Elliott's computation of the $K$-theory and trace pairing of $A_\theta$ \cite{Ell84} (see also \cite{Rie88}). 

\begin{thm}[{\cite{Ell84}}]\label{T-Ell}
	Let $\theta$ be a skew-symmetric real $d\times d$ matrix. Then, for $i\in \{0,1\}$,  
	\[
		K_i(A_\theta) \cong \Zb^{2^{d-1}}.	
	\]
	Moreover, there is an isomorphism from $K_0(A_\theta)$ to the even exterior algebra $\Lambda^{\mathrm{even}}\Zb^d$ sending $[1_{A_{\theta}}]_0$ to $1\in \Lambda^0\Zb^d = \Zb$ such that the induced map $(\tau_\theta)_*$ is equal to the exterior exponential $\exp(\theta) : \Lambda^{\mathrm{even}}\Zb^d \to \Rb$, where $\tau_\theta$ denotes the canonical trace on $A_\theta$. 
\end{thm}

Rather than defining the exterior exponential here (see \cite{Ell84,Rie88} for the details), we provide a more explicit interpretation of Theorem~\ref{T-Ell} in the context of those $A_\theta$ arising from free minimal actions. We encountered this interpretation in \cite[Theorem 5.7.1]{ProSch16} (see also \cite[Proposition 4.1]{BenMat18}). By forgetting the graded algebra structure on $\Lambda^{\mathrm{even}}\Zb^{d+1}$, we have  
\[
	\Lambda^{\mathrm{even}}\Zb^{d+1} = \bigoplus_{\substack{I\subseteq \{1,\ldots,d+1\}, \\ |I| \text{ even}}} \Zb e_{I},  
\]
where for $I = \{i_1,\ldots,i_k\}$ with $i_1 < \cdots < i_k$, we define $e_I := e_{i_1} \wedge \cdots \wedge e_{i_k}$ and $e_{\emptyset} = 1 \in \Lambda^0\Zb^{d+1}$. By the binomial theorem, the number of (possibly empty) increasingly ordered subsets $I$ of $\{1,\ldots,d+1\}$ with even cardinality is
\[
\sum_{k \text{ even}} {{d+1}\choose{k}} = 2^d,
\] 
so these subsets parametrise a basis of $\Lambda^{\mathrm{even}}\Zb^{d+1}$. Fixing an ordering of the even-cardinality increasingly ordered subsets $I$ yields an identification $\Lambda^{\mathrm{even}}\Zb^{d+1} \cong \Zb^{2^d}$ of abelian groups. Recall that for a $2n\times 2n$ skew-symmetric matrix $\theta$, the Pfaffian of $\theta$, denoted $\operatorname{pf}(\theta)$, is defined as the real number 
\[
	\operatorname{pf}(\theta) = \frac{1}{2^n n!} \sum_{\sigma \in S_{2n}} \operatorname{sgn}(\sigma) \prod_{i=1}^n \theta_{\sigma(2i-1),\sigma(2i)}.
\]
Now Elliott's trace pairing formula may be expressed on basis elements as 
\begin{equation}
	(\tau_\theta)_*(e_I) = \operatorname{pf}(\theta_I),
\end{equation}
where $\operatorname{pf}(\theta_{\emptyset}) := 1$ and for a non-empty increasing subset $I\subseteq \{1,\ldots,d+1\}$, $\theta_I$ denotes the submatrix of $\theta$ obtained by restricting indices to $I$. For a skew-symmetric matrix of the form (\ref{E-torus}), $\operatorname{pf}(\theta_I)$ vanishes for all subsets $I$ of even cardinality unless $I = \emptyset$ or $I = \{1,i\}$, where $i\in \{2,\ldots,d+1\}$. In this case, we have 
\[
	\theta_{\{1,i\}} = \begin{pmatrix}
		0 & \gamma_i\\
		-\gamma_i & 0
	\end{pmatrix}, \quad \operatorname{pf}(\theta_{\{1,i\}}) = \gamma_i.
\]
Thus by ordering the basis of $\Lambda^{\mathrm{even}}\Zb^{d+1}$ so that $\emptyset, \{1,2\}, \ldots, \{1,d+1\}$ are listed first, and letting $\Psi : \Lambda^{\mathrm{even}}\Zb^{d+1} \to \Zb^{2^d}$ be the associated isomorphism, we obtain 
\begin{equation}
	(\tau_\theta)_* \circ \Psi^{-1}(n_0,\ldots,n_{2^d-1}) = n_0 + \sum_{i=1}^d \gamma_i n_i.
\end{equation}

We summarise the preceding discussion for later use.   

\begin{cor}\label{C-trace}
	Let $\theta$ be a skew-symmetric real $(d+1)\times (d+1)$ matrix of the form (\ref{E-torus}) and $\tau_\theta$ the canonical trace on $A_\theta$. Then there is an isomorphism $K_0(A_\theta)\cong \Zb^{2^d}$ sending $[1_{A_\theta}]_0$ to the first standard basis vector, with respect to which the trace pairing is given by  
	\[
	(\tau_\theta)_*(n_0,n_1,\ldots,n_{2^{d}-1}) = n_0 + \sum_{i=1}^d \gamma_i n_i.
	\]
In particular, $(\tau_\theta)_*(K_0(A_\theta)) = \Zb + \gamma_1\Zb + \cdots + \gamma_d\Zb$. 
\end{cor}

The following key observation generalises \cite[Proposition 5.4]{PutSchSka86}.

\begin{prop}\label{P-embed}
	Let $\Zb^d \curvearrowright \Tb$ be a free Denjoy action with $\gamma_i$ the rotation number of the standard basis vector $e_i \in \Zb^d$. Then there is a canonical embedding $\Phi : A_\theta \hookrightarrow C(\Tb)\rtimes \Zb^d$, where $A_\theta$ denotes the noncommutative torus associated with the matrix (\ref{E-torus}). Moreover, the unique trace $\tau$ on $C(\Tb)\rtimes \Zb^d$ satisfies $\tau \circ \Phi = \tau_{\theta}$, where $\tau_{\theta}$ is the canonical trace on $A_\theta$. 
\end{prop}

\begin{proof}
	By Proposition~\ref{P-amen}, there exists a semiconjugacy $\varphi : \Tb \to \Tb$ that intertwines the Denjoy action, which we denote by $\alpha$, with the rotation action $\beta : g \mapsto R_{\rho(g)}$. Since $\varphi$ is continuous and surjective, the pullback $\varphi^* : C(\Tb) \to C(\Tb)$ given by $f\mapsto f\circ \varphi$ is an injective $*$-homomorphism. By equivariance of $\varphi$, we have $\alpha(g)\circ \varphi^* = \varphi^*\circ \beta(g)$ for all $g\in G$. Hence $\varphi^*$ extends to an embedding $\Phi : C(\Tb)\rtimes_{\beta} \Zb^d \cong A_\theta \hookrightarrow C(\Tb)\rtimes_{\alpha} \Zb^d$ determined by $f\tilde{\lambda}_g \mapsto \varphi^*(f)\lambda_g$, for $f\in C(\Tb)$ and $g\in \Zb^d$, where $\lambda_g$ and $\tilde{\lambda}_g$ are the unitaries implementing the actions $\alpha(g)$ and $\beta(g)$, respectively. Moreover, for $f \in C(\Tb)$ and $g\in \Zb^d$, we use Corollary~\ref{C-trace} and Remark~\ref{R-meas} to compute, 
	\begin{align*}
		\tau(\varphi^*(f)\lambda_g) &= \begin{cases}
		\displaystyle \int_{\Tb} \varphi^*(f) \, d\mu, & g = e\\
		0, & \text{otherwise} 
		\end{cases}\\
		& = \begin{cases}
		\displaystyle \int_{\Tb} f\, d\lambda, & g = e\\
		0, & \text{otherwise} 
		\end{cases}\\
		&= \tau_{\theta}(f\tilde{\lambda}_g),
	\end{align*}
	where $\lambda$ denotes the Lebesgue measure on $\Tb$. Since elements of the form $\varphi^*(f)\lambda_g$ span a dense subalgebra of $\Phi(A_\theta)$, we have $\tau \circ \Phi = \tau_{\theta}$, which completes the proof.
\end{proof}

The $K$-groups $K_0(C(\Tb))$ and $K_1(C(\Tb))$ are generated by $[1]_0$ and $[u]_1$, respectively, where $u : z\mapsto z$ is the unitary generating $C(\Tb)$. There is a natural inclusion $C(\Tb)\hookrightarrow C(\Tb)\rtimes \Zb^d$ given by $f \mapsto f\lambda_e$. We will sometimes abuse notation and identify $f\lambda_e$ with $f$; in particular, we treat $[1]_0$ and $[u]_1$ as elements of $K_*(C(\Tb)\rtimes \Zb^d)$. 

\begin{thm}\label{T-Kthy}
	Fix $d\in \Nb$. Let $\Zb^d \curvearrowright \Tb$ be a free Denjoy action, and let $A = C(\Tb) \rtimes \Zb^d$. Then for $i\in\{0,1\}$, the embedding $\Phi : A_\theta \hookrightarrow A$ as in Proposition~\ref{P-embed} induces a group isomorphism $\Phi_{*,i} : K_i(A_\theta) \to K_i(A)$. Consequently, we have $K_i(A) \cong \Zb^{2^d}$ for $i\in \{0,1\}$.
\end{thm}

\begin{proof}
	The final statement follows from Theorem~\ref{T-Ell}, thus it suffices to show that the induced maps $\Phi_{*,i}$ are isomorphisms for $i \in \{0,1\}$. Write $\alpha$ for the free Denjoy action $\Zb^d \curvearrowright \Tb$ and $\beta$ for the minimal rotation action $g\mapsto R_{\rho(g)}$ to which it is semiconjugate (Proposition~\ref{P-amen}). By definition, there is a semiconjugacy $\varphi : \Tb \to \Tb$, and the pullback $\varphi^* : C(\Tb) \to C(\Tb)$ is equivariant for the induced actions on $C(\Tb)$, i.e., $\varphi^* \circ \beta(g) = \alpha(g) \circ \varphi^*$ for all $g\in \Zb^d$. Let $H_0 := \{e\}$ be the trivial subgroup of $\Zb^d$ and, for $k \in \{1,\ldots,d\}$, let $H_k := \langle e_1,\ldots,e_k\rangle \cong \Zb^k$ be the subgroup generated by the first $k$ standard basis elements. We define, for $k\in \{0,\ldots,d\}$,
	\begin{equation}
		A_k := C(\Tb) \rtimes_{\alpha|_{H_k}} \Zb^k, \quad B_k := C(\Tb) \rtimes_{\beta|_{H_k}} \Zb^k.
	\end{equation}
Both $A_0$ and $B_0$ are canonically identified with $C(\Tb)$, $A_d = A$, and $B_d = A_\theta$. Since $\varphi^*$ is equivariant, it is also equivariant for each of the restricted actions $H_k \curvearrowright C(\Tb)$ so for each $k\in \{0,\ldots,d\}$ we obtain injective $*$-homomorphisms $\Phi_k : B_k \hookrightarrow A_k$ determined by $\Phi_k(f\tilde{\lambda}_g) = \varphi^*(f)\lambda_g$ for $f\in C(\Tb), g\in H_k$, where $\tilde{\lambda}_g$ and $\lambda_g$ denote the canonical implementing unitaries in $B_k$ and $A_k$, respectively. For $k = d$, this is precisely the embedding $\Phi : A_\theta\hookrightarrow A$ from Proposition~\ref{P-embed}.   
	
	Since $H_k = H_{k-1} \oplus \Zb e_k$ for $k \in \{1,\ldots,d\}$, we may also identify 
	\begin{equation}\label{E-iterated}
	A_k \cong A_{k-1} \rtimes_{\alpha_k} \Zb, \quad B_k \cong B_{k-1} \rtimes_{\beta_k} \Zb,
	\end{equation}
where $\alpha_k$ and $\beta_k$ are the automorphisms of $A_{k-1}$ and $B_{k-1}$ respectively induced by the actions of $e_k$. Explicitly, $\alpha_k$ extends $\alpha(e_k)$ on $C(\Tb)$ and fixes $\lambda_{e_i}$ for $i < k$, and similarly for $\beta_k$. These automorphisms are well-defined since the unitaries pairwise commute. For $f\in C(\Tb)$ and $h\in H_{k-1}$, the isomorphisms in (\ref{E-iterated}) are given explicitly by, 
\[
	f\lambda_{h+ne_k} \mapsto (f\lambda_h) v_k^n, \quad  f\tilde\lambda_{h+ne_k} \mapsto (f\tilde\lambda_h) w_k^n, 
\]
where $v_k$ and $w_k$ denote the implementing unitaries corresponding to the $\Zb$-actions $\alpha_k$ and $\beta_k$, respectively. Using this picture, we see that the $\Phi_k$ are compatible with one another in the following sense. The equivariance relation $\Phi_{k-1} \circ \beta_k = \alpha_k \circ \Phi_{k-1}$ yields, by the universal property of the crossed product, a homomorphism $\Phi_{k-1}\rtimes \Zb : B_{k-1} \rtimes_{\beta_k} \Zb \longrightarrow A_{k-1} \rtimes_{\alpha_k} \Zb$ satisfying $\Phi_{k-1}\rtimes \Zb((f\tilde{\lambda}_h) w_k^n) = (\varphi^*(f)\lambda_h) v_k^n$ for $f\in C(\Tb), h\in H_{k-1}, n\in \Zb$. Thus, under the identifications (\ref{E-iterated}), we have $\Phi_k = \Phi_{k-1} \rtimes \Zb$.  
	
	We prove by a finite induction on $k$ that the induced maps 
	\[
		(\Phi_k)_{*,i} : K_i(B_k) \longrightarrow K_i(A_k)
	\]
	are isomorphisms for $i\in \{0,1\}$ for all $k \in \{0,\ldots,d\}$. For $k = 0$, observe that $\varphi^*$ is unital so $(\Phi_0)_{*,0} = (\varphi^*)_{*,0}$ fixes the generator $[1]_0$ of $K_0(C(\Tb))$. Moreover, since $\varphi$ is degree one, we have 
	\[(\Phi_0)_{*,1}([u]_1) = (\varphi^*)_{*,1}([u]_1) = [u\circ \varphi]_1 = [u]_1.\] 
	Hence $(\Phi_0)_{*,i}$ is an isomorphism for $i \in \{0,1\}$. Now fix $k\in \{1,\ldots,d\}$ and suppose that $(\Phi_{k-1})_{*,i}$ is an isomorphism for $i\in \{0,1\}$. The rotation $\beta({e_k})$ is connected to the identity automorphism of $C(\Tb)$ through the path of rotations $R_{t\rho(e_k)}$, where $t\in [0,1]$. Since rotations commute, this extends to a homotopy of automorphisms of $B_{k-1}$ from the identity to $\beta_k$. Thus by homotopy invariance of $K$-theory, we have $(\beta_k)_{*,i} = \mathrm{id}$, $i\in\{0,1\}$.  

Since $\Phi_{k-1} \circ \beta_k = \alpha_k \circ \Phi_{k-1}$, by functoriality we have, for $i\in \{0,1\}$,
\[
	(\Phi_{k-1})_{*,i} \circ (\beta_k)_{*,i} = (\alpha_k)_{*,i} \circ (\Phi_{k-1})_{*,i}. 
\]
By assumption, $(\Phi_{k-1})_{*,i}$ is an isomorphism, and $(\beta_k)_{*,i}$ is the identity and thus $(\alpha_k)_{*,i} = (\Phi_{k-1})_{*,i} \circ (\beta_k)_{*,i} \circ (\Phi_{k-1})_{*,i}^{-1} = \mathrm{id}$. Since both $\beta_k$ and $\alpha_k$ act trivially on $K$-theory, applying the Pimsner--Voiculescu six-term exact sequence yields, for $i\in \{0,1\}$, short exact sequences 
\[
	0 \longrightarrow K_i(B_{k-1}) \longrightarrow K_i(B_k) \longrightarrow K_{1-i}(B_{k-1}) \longrightarrow 0
\]
and 
\[
	0 \longrightarrow K_i(A_{k-1}) \longrightarrow K_i(A_k) \longrightarrow K_{1-i}(A_{k-1}) \longrightarrow 0.
\]
By naturality of the Pimsner--Voiculescu six-term exact sequence applied to the equivariant map $\Phi_{k-1} : B_{k-1} \to A_{k-1}$, we obtain commutative diagrams
\[\begin{tikzcd}
	0 & K_i(B_{k-1}) & {K_i(B_k)} & K_{1-i}(B_{k-1}) & 0 \\
	0 & K_i(A_{k-1}) & {K_i(A_k)} & K_{1-i}(A_{k-1}) & 0
	\arrow[from=1-1, to=1-2]
	\arrow[from=1-2, to=1-3]
	\arrow["{(\Phi_{k-1})_{*,i}}", from=1-2, to=2-2]
	\arrow[from=1-3, to=1-4]
	\arrow["{(\Phi_k)_{*,i}}", from=1-3, to=2-3]
	\arrow[from=1-4, to=1-5]
	\arrow["{(\Phi_{k-1})_{*,1-i}}", from=1-4, to=2-4]
	\arrow[from=2-1, to=2-2]
	\arrow[from=2-2, to=2-3]
	\arrow[from=2-3, to=2-4]
	\arrow[from=2-4, to=2-5]
\end{tikzcd}
\]
By assumption, the left and right vertical maps are isomorphisms, and hence the middle vertical maps $(\Phi_k)_{*,i}$, which are canonically identified with $(\Phi_{k-1}\rtimes \Zb)_{*,i}$, are also isomorphisms, by the five lemma. This completes the induction and, since $\Phi_d = \Phi$ as in Proposition~\ref{P-embed}, the proof.     
\end{proof}

\begin{thm}\label{T-Kthy2}
	Fix $d\in \Nb$. Let $\Zb^d \curvearrowright \Tb$ be a free Denjoy action with $\gamma_i$ the rotation number of the standard basis vector $e_i \in \Zb^d$, and let $A = C(\Tb) \rtimes \Zb^d$. Denoting by $\tau$ the unique trace on $A$, there is an identification $K_0(A)\cong \Zb^{2^d}$ sending $[1_A]_0$ to the first standard basis vector, with respect to which the trace pairing $\tau_* : K_0(A) \to \Rb$ is given by 
	 \[
		\tau_*(n_0,n_1,\ldots,n_{2^d - 1}) = n_0 + \sum_{i=1}^d n_i \gamma_i.	
	\]
In particular,
	\[
		\tau_*(K_0(A)) = \Zb + \gamma_1\Zb + \cdots + \gamma_d \Zb.
	\] 
\end{thm}

\begin{proof}
Let \(\Phi:A_\theta\hookrightarrow A\) denote the embedding from
Proposition 3.22. By Theorem 3.23, \(\Phi\) induces an isomorphism $K_0(A_\theta)\cong K_0(A)$. Moreover, Proposition 3.22 gives $\tau\circ \Phi=\tau_\theta$, hence the diagram
\begin{equation}\label{E-comm}
\begin{tikzcd}
	{K_0(A_\theta)} & {K_0(A)} \\
	\Rb & \Rb
	\arrow["{\cong}", "{\Phi_{*,0}}"', from=1-1, to=1-2]
	\arrow["{(\tau_{\theta})_*}"', from=1-1, to=2-1]
	\arrow["{\tau_*}", from=1-2, to=2-2]
	\arrow[equal, from=2-1, to=2-2]
\end{tikzcd}
\end{equation}
commutes. Let $\Psi:K_0(A_\theta) \to \Zb^{2^d}$ be the identification from Corollary~\ref{C-trace}, chosen so that $\Psi([1_{A_\theta}]_0)$ is the first standard basis vector and the trace pairing is given by 
\[
	(\tau_\theta)_* \circ \Psi^{-1}(n_0,\ldots,n_{2^d-1}) = n_0 + \sum_{i=1}^d n_i \gamma_i.
\]
We identify $K_0(A)$ with $\Zb^{2^d}$ via the group isomorphism $\Psi \circ \Phi_{*,0}^{-1}$. Since $\Phi$ is unital, $\Phi_{*,0}([1_{A_\theta}]_0) = [1_A]_0$ and hence $[1_A]_0$ is sent to the first standard basis vector. The desired formula for $\tau_*$ now follows from the commutative diagram (\ref{E-comm}) and Corollary~\ref{C-trace}. The description of $\tau_*(K_0(A))$ follows immediately. 
\end{proof}

\begin{rem}\label{R-gen}
	It is immediate from Theorem~\ref{T-Kthy} that any fixed choice of bases
	for $K_i(A_\theta)$, $i\in\{0,1\}$, is transported via the inclusion
	$A_\theta\hookrightarrow A$ to bases for the $K$-groups of
	$A:= C(\Tb)\rtimes \Zb^d$. In particular, using Rieffel's work
	\cite{Rie88}, one may choose a basis of $K_0(A)$ consisting of
	$[1_A]_0$, the classes of Rieffel projections $p_i\in A$ satisfying
	$\tau(p_i)=\gamma_i$ for $i\in\{1,\ldots,d\}$, and $2^d-d-1$ Bott classes 			$x\in K_0(A)$ satisfying $\tau_*(x)=0$. The latter generators roughly 				correspond to tuples of commuting generating unitaries in $A$. As first				observed by Loring \cite{Lor88}, such elements can be viewed as classes of 			generalised virtual Rieffel projections in $M_2(A)$.

	As explained in \cite{PutSchSka86}, it is also possible to find Rieffel
	projections in $C(\Tb)\rtimes \Zb^d$ representing the relevant
	$K_0$-classes which do not belong to the embedded copy
	$A_\theta\subseteq C(\Tb)\rtimes \Zb^d$. This amounts to choosing
	functions in the construction of Rieffel projections
	which are not constant on the intervals of $\Tb\setminus Y$. We refer
	the reader to \cite[\S 2]{Bar16} for further details about Bott classes,
	including their connection to determining generators of $K_1(A)$.
\end{rem}

\begin{rem}
	From Theorem~\ref{T-Kthy2}, we see that given a free Denjoy action $\Zb^d \curvearrowright \Tb$, the rotation numbers $\gamma_i$ corresponding to each generator may be recovered from the $K$-theory of $C(\Tb) \rtimes \Zb^d$ by means of the trace pairing. This is certainly necessary to obtain classification, though much more information must be recovered; compare with \cite{PutSchSka86}, where the invariant $Q(\alpha)$ of a Denjoy homeomorphism $\alpha$ is recovered using the six-term exact sequence in $K$-theory coming from (\ref{exact_2}).
\end{rem}

\end{document}